\documentclass{amsart}

\usepackage[latin1]{inputenc}
\usepackage{graphicx}
\usepackage{psfrag}
\usepackage{amsthm,amsfonts,amssymb,amscd,amsmath}

\theoremstyle{plain}
  \newtheorem{theorem}{Theorem}[section]
  \newtheorem*{mainthm}{Main Theorem}
  \newtheorem{lemma}[theorem]{Lemma}
  \newtheorem{prop}[theorem]{Proposition}

\theoremstyle{definition}
  \newtheorem{definition}[theorem]{Definition}

\newcommand{\field}[1]{\mathbb{#1}}

\def\prho#1{p_{1\over \rho}\left(#1\right)}

\def\fC{\field{C}}

\def\fR{\field{R}}

\def\cD{\mathcal{D}}

\def\cT{\mathcal{T}}

\def\cS{\mathcal{S}}

\date{\today}

\begin{document}

\title[Existence of a Lorenz renormalization fixed point]{Existence of a Lorenz
renormalization fixed point of an arbitrary critical order}
\author{Denis Gaidashev, Bj\"orn Winckler}

\begin{abstract}
  We present a proof of the existence of a renormalization fixed point for
  Lorenz maps of the simplest non-unimodal combinatorial type
  $(\{0,1\},\{1,0,0\})$ and with a critical point of arbitrary order $\rho>1$.
\end{abstract}
\maketitle

\setcounter{page}{1}


\section{Introduction}

E. N. Lorenz in \cite{Lor} demonstrated  numerically the existence of certain three-dimensional flows that have a complicated behavior. The {\it Lorenz flow} has a saddle fixed point with a one-dimensional unstable manifold and an infinite set of periodic orbits whose closure constitutes a global attractor of the flow.

As it is often done in dynamics, one can attempt to understand the behaviour of a three-dimensional flow by looking at the first return map to an appropriately chosen two-dimensional section. In the case of the Lorenz flow, it is convenient to chose the section as a plane transversal to the local stable manifold, and, therefore,  intersecting it along a curve $\gamma$. The first return map is discontinuous at $\gamma$.

The {\it geometric Lorenz flow} has been introduced in \cite{Wil}: a Lorenz flow with an extra condition that the return map preserves a one-dimensional foliation in the section, and contracts distances between points in the leafs of this foliation at a geometric rate. Since the return maps is contracting in the leafs, its dynamics is asymptotically one-dimensional, and can be understood in terms of a map acting on the space of leafs (an interval). This interval map has a discontinuity  at the point of the interval corresponding to $\gamma$, and is commonly called  the {\it Lorenz maps}. More precisely,

\begin{definition} \label{Lorenz_def}
Let $s>0$ and $\rho > 0$. A $C^{s}$-Lorenz map $\psi: [-1,r] \mapsto [-1,r]$ is a map given by a pair $(f,g)$, such that:
\begin{itemize}
\item[1)] $f:[-1,0)\mapsto [-1,r]$ and $g: (0,r] \mapsto [-1,r]$. $f$ and $g$ are continuous and strictly increasing;
\item[2)] there exists $\rho>0$, the exponent of $\psi$, such that
$$f(x)=l(|x|^\rho), \quad g(x)=t(|x|^\rho),$$
$l$ and $t$ being  $C^s$-diffeomorphisms.
\end{itemize}
\end{definition}

Guckenheimer and Williams have proved in \cite{GW} that there is an open set of three-dimensional vector fields, that generate a geometric Lorenz flow with a smooth Lorenz map of $\rho<1$. However, one can use the arguments of \cite{GW} to construct open sets of vector fields with Lorenz maps of $\rho \ge 1$.  Similarly to the unimodal family, Lorenz maps with $\rho > 1$ have a richer dynamics that combines contraction with  expansion.

For any $x \in [-1,r] \setminus \{0\}$ such that $f^n(x) \ne 0$ for all $n \in \field{N}$, define the itinerary $\omega(x) \in \{0,1\}^{\field{N}}$ of $x$ as the sequence $\{\omega_0(x), \omega_1(x),\ldots\}$,  such that
$$ \omega_i = \left\{0, \quad f^i(x)<0, \atop 1, \quad f^i(x)>0.\right.$$
If one imposes the usual  order $0<1$, then for any two $\omega$ and $\tilde{\omega}$ in $\{0,1\}^{\field{N}}$  we say that $\omega<\tilde{\omega}$ iff there exists $r \ge 0$ such that $\omega_i = \tilde{\omega}_i$ for all $i<r$ and $\omega_r<\tilde{\omega}_r$.

The limits 
$$\omega(x^+)=\lim_{y \rightarrow x^+}  \omega(y), \quad \omega(x^-)=\lim_{y \rightarrow x^-}  \omega(y),$$
where $y$ runs through points which are not preimages of $0$, exists for all $y \in [-1,r]$.

{\it The kneading invariant} $K(f)$ of $f$ is the pair $(K^-(f),K^+(f))=(\omega(0^-),\omega(0^+))$. Hubbard and Sparrow have shown in \cite{HS} that $(K^-,K^+)$ is the kneading invariant of some topologically expansive Lorenz map iff for all  $n \in \field{N}$

$$K^-_0=0,  \quad   K^+_0=1,  \quad   \sigma(K^+) \le \sigma^n(K^+) < \sigma(K^-),   \quad  \sigma(K^+) < \sigma^n (K^-) \le \sigma(K^-),$$
here $\sigma$ is the shift in $\{0,1\}^{\field{N}}$.

Kneading invariants for a general Lorenz map, not necessarily expansive,  satisfy a weaker condition:
$$K^-_0=0, \quad K^+_0=1, \quad \sigma(K^+) \le \sigma^n(K^\pm) \le  \sigma(K^-), \quad n \in \field{N}.$$
Conversely, any sequence as above is a kneading sequence for some Lorenz  map.

A Lorenz map $f$ is called {\it renormalizable} if there exist $p$ and $q$, $-1<p<0<q<r$, such that the first return map $(f^n,g^m)$, $n>1, m>1$, of $[p,q]$ is a Lorenz map.

The intervals $f^i([p,0))$, $1 \le i \le n-1$, are pairwise disjoint, and disjoint from $[p,q]$. So are the intervals,  $f^i((0,q])$, $1 \le i \le m-1$. Since these intervals do not contain zero, we can associate a finite sequence of $0$ and $1$ to each sequence of the intervals:
$$K^-=\{K^-_0,\ldots,K^-_{n-1}  \}, \quad  K^+=\{K^+_0,\ldots,K^+_{m-1}\},$$ 
which will be called the {type of renormalization}. The subset of maps $\ref{Lorenz_def}$ which are renormalizable of type $(\alpha,\beta)$ is referred to as the domain of renormalization $\cD_{\alpha,\beta}$ (cf.~\cite{MM}).

The study of renormalizable Lorenz maps was initiated by Tresser et al. (see
e.g.~\cite{CCT}) but a more recent paper is that of Martens and de Melo
(see~\cite{MM}). The latter authors consider the combinatorics of the
renormalizable maps, and prove several results about the domains of
renormalization and the structure of the parameter plane for two-dimensional
Lorenz  families.

The second author of the present paper has provided a computer assisted proof
of existence of a renormalization fixed point for the renormalization operator
of type  $(\{0,1\},\{1,0,0\})$ in \cite{Win1}. Furthermore, issues of existence
of renormalization periodic points and hyperbolicity have been addressed by
the second author in \cite{Win2}, where it is proved that the limit set of
renormalization, restricted to monotone combinatorics with the return time of
one branch being large, is a Cantor set, and that each point in the limit set
has a two-dimensional unstable manifold. This result holds for any real
$\rho>1$.

In this paper we give an analytic proof of the result of \cite{Win1} for a general exponent of the Lorenz map $\rho>1$.

We consider the renormalization operator $R$ of type $(\alpha,\beta)=(\{0,1\},\{1,0,0\})$, specifically $R(f,g)=(\hat{f},\hat{g})$ where 
\begin{align}
\label{eq1}  \hat{f}(z)&= \lambda^{-1} g(f(\lambda z)),\\
\label{eq2}   \hat{g}(z)&= \lambda^{-1} f(f(g(\lambda z))),\\
\label{eq3}  \lambda&= -f(f(-1)).
\end{align}

As usual, the notation $C^\omega$ will denote the analytic class of maps.

\begin{mainthm}
For every $\rho > 1$, there exists a $C^\omega$-Lorenz map $(f^*,g^*)$ which is a fixed point of the renormalization of type $(\{0,1\},\{1,0,0\})$.
\end{mainthm}

\begin{figure}
 \begin{center}
{\includegraphics[height=70mm,width=70mm,angle=-90]{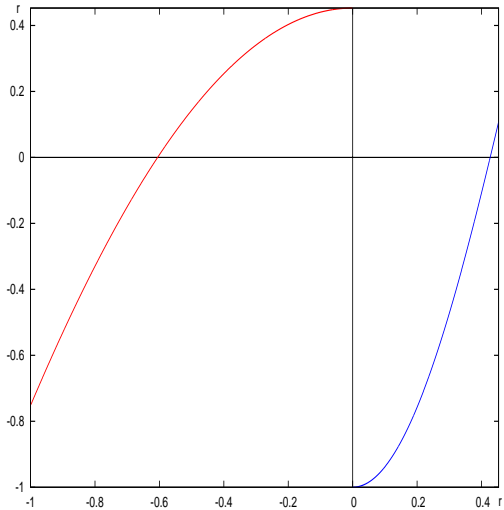}}
\caption{ \it The renormalization fixed  point for $\rho=2$ computed in \cite{Win1}, r=0.453 \ldots}
\end{center}
\end{figure}

\medskip

To prove the theorem we introduce an operator on an appropriate functional space of the diffeomorphic parts of the inverse branches of $f$ and $g$. The crucial ingredient of  our proof is a demonstration that there exists a subset in this functional space, invariant under the operator, characterized by the condition that the nonlinearities of the inverse branches are negative and bounded away from zero. It is this negativity of the nonlinearity that seems to be indispensable to complete the proof. 

We would like to remark that the results of this paper could be made somewhat more general: indeed a similar method can be used to demonstrate existence of renormalization fixed points of  other types for longer $\alpha$ and $\beta$. This is the range of $\alpha$ and $\beta$ that was not accessible through the methods used  in \cite{Win2} where the condition that one of the branches has a very long return time (long $\alpha$ or $\beta$) was crucial. 

However, we believe that at this point it would be timely to attempt to built a complete renormalization theory for Lorenz maps that would  mirror that for unimodal maps. Specifically, one could attempt to extend the results of \cite{Win2} to all return times, and  demonstrate existence of the whole renormalization horseshoe via real or complex a priori bounds.  We believe, that the negativity  of the nonlinearity of the inverse branches could again play an important role in such proofs.

\section{An operator on the Epstein class} \label{Epstein}

Consider the action of this operator on the {\it little Epstein class} of functions, that is, functions $f$ and $g$ factorizable as $f=l \circ p_\rho \circ -id$, and $g=t \circ p_\rho$, $\rho>1$, where  $p_\rho$ is the exponential map
$$p_\rho(z)=z^{\rho},$$
and $l$, $t$ are some diffeomorphisms (to be specified later) of the range of
$p_\rho$ (cf.\ \cite{Eps1}, \cite{Eps2}, \cite{Eps3}). Ignoring the issue of domains of maps for a moment, we get for the fixed point version of $(\ref{eq1})$-$(\ref{eq2})$. 
\begin{align} 
\nonumber l \circ p_\rho \circ -id&= \lambda^{-1} \circ  t \circ p_\rho \circ l \circ p_\rho \circ -\lambda \\
\nonumber \lambda \circ l \circ p_\rho&= t \circ p_\rho \circ l \circ p_\rho \circ\lambda \\
\nonumber t^{-1} \circ \lambda  \circ l \circ p_\rho &= p_\rho \circ l \circ p_\rho \circ \lambda \\
\nonumber p_{1\over \rho} \circ t^{-1} \circ \lambda \circ l \circ p_\rho &=  l \circ p_\rho \circ \lambda \\
\nonumber l^{-1} \circ p_{1\over \rho} \circ t^{-1} \circ \lambda \circ l \circ p_\rho &=  p_\rho \circ \lambda \\
\nonumber l^{-1} \circ p_{1\over \rho} \circ t^{-1} \circ \lambda &= p_\rho  \circ 
\lambda \circ p_{1\over \rho} \circ l^{-1}\\
\label{new_eq1} l^{-1} \circ p_{1\over \rho} \circ t^{-1} \circ \lambda &= \lambda^\rho \circ l^{-1},
\end{align}
here, $p_{1\over \rho}$ is the root function 
$$p_{1\over \rho}(r e^{i \theta})=r^{1 \over \rho} e^{i \theta \over \rho}.$$

In a similar way, we get from $(\ref{eq2})$ the following equation for inverse diffeomorphic parts $l^{-1}$ and $t^{-1}$ of the inverse branches of the fixed point of $R$: 
\begin{equation}\label{new_eq2}
t^{-1} \circ -p_{1\over \rho} \circ l^{-1} \circ -p_{1\over \rho} \circ l^{-1} \lambda=\lambda^\rho \circ t^{-1}. 
\end{equation}

Define diffeomorphisms $U$ and $V$ by setting
$$l^{-1}(z)=a U(r-z), \quad t^{-1}(z)=b V(z+1),$$ 
where the normalizing constants $a$ and $b$ will be chosen below. $U$ and $V$ are defined on\begin{equation} 
\label{UVdomains} [0,r-f(-1)]   \quad {\rm and} \quad  [0,g(0)+1],
\end{equation}
 respectively. Then $(\ref{new_eq1})$ and $(\ref{new_eq2})$ become
\begin{align*}
\lambda^\rho U(r-z)&=U\left(r-\prho{b V(\lambda z+1)}\right),  \\ 
\lambda^\rho V(z+1)&=V\left(1-\prho{a U \left(r+\prho{aU(r-\lambda z)}\right) } \right), \end{align*}
or 
\begin{align}
\label{eqq1} \lambda^\rho U(z)&=U\left(r-\prho{b V(\lambda (r-z)+1)}\right),  \\ 
\label{eqq2} \lambda^\rho V(z)&=V\left(1-\prho{a U \left(r+\prho{aU(r-\lambda (z-1))}\right) } \right). 
\end{align}

Set 
\begin{equation}\label{ab}
b={ r^\rho \over V(\lambda r+1) }, \quad {\rm and} \quad a={1 \over U(r+y)},
\end{equation}
where $y=y(\lambda)$ solves
\begin{equation}\label{y}
y=\prho{U(r+\lambda) \over U(r+y)}.  
\end{equation}
We will demonstrate in  Section $\ref{decoupled}$ that $(\ref{y})$ has a unique solution whenever $U$ is an appropriate functional class.

Equations $(\ref{eqq1})$ and $(\ref{eqq2})$ become
\begin{align}
\label{new_eqq1} \lambda^\rho U &= U \circ \Psi_{V,\lambda,r},&
\Psi_{V,\lambda,r} &\equiv r-r \prho{{ V(\lambda (r-z)+1) \over  V(\lambda r+1)}},  \\ 
\label{new_eqq2} \lambda^\rho V &= V \circ \Phi_{U,\lambda,r},&
\Phi_{U,\lambda,r} &\equiv 1-\prho{a U \left(r+\prho{aU(r-\lambda (z-1))}\right) }.
\end{align}

Notice, that the  normalization constants $a$ and $b$ have been chosen so that 
$$\Psi_{V,\lambda,r}(0)=\Phi_{U,\lambda,r}(0)=0.$$



At this point we will ``decouple'' the system $(\ref{new_eqq1})-(\ref{new_eqq2})$, i.e. we will allow  the scaling $\lambda$ in $(\ref{leq})$ and $(\ref{mueq})$ to be two independent quantities.  These new scaling parameters will be called $\lambda$ and $\mu$. At the intuitive level, we are introducing more freedom in the system, and changing the problem of a search four a quadruple $(U,V,\lambda,r)$ from a pair of fixed point equations into a problem of looking for pairs $(U,\lambda)$ and $(V,\mu)$ independently from separate equations while keeping $r$ as a parameter. The second is a more accessible problem. We will demonstrate that  solutions of the decoupled fixed point equations exists, and that $\lambda$ and $\mu$, corresponding to fixed points, are  continuous in the parameter $r$. At the end, the value of $r$ will be adjusted so that $\lambda=\mu$, providing a solution to the original problem. The decoupled system takes the following form.
\begin{align}
\label{leq}  \lambda^\rho U &= U \circ \Psi_{V,\lambda,r},& \Psi_{V,\lambda,r} &\equiv r-r \prho{ V(\lambda (r-z)+1) \over  V(\lambda r+1)},  \\ 
\label{mueq}  \mu^\rho V&= V \circ \Phi_{U,\mu,r},& \Phi_{U,\mu,r} &\equiv 1-\prho{U \left(r+\prho{ U(r-\mu (z-1)) \over U(r+y)}\right) \over U(r+y)  }, \\
 \label{y_mu} & & y &= p_{1 \over \rho}\left(U(r+\mu) \over U(r+y)\right).
\end{align}

Set
\begin{equation}
\label{zw} Z(z)=\prho{ U(r+z) \over U(r+y) }, \quad  W(z)=\prho{ V(z+1) \over V(\lambda  r+1) },
\end{equation}
with this notation the decoupled system becomes
\begin{align}
\label{dec_eq1}  \lambda^\rho U &= U \circ \Psi_{V,\lambda,r},&
\Psi_{V,\lambda,r}(z) &\equiv r-r W(\lambda(r-z)),  \\ 
\label{dec_eq2}  \mu^\rho V&= V \circ \Phi_{U,\mu,r},& \Phi_{U,\mu,r}(z)
&\equiv 1-Z(Z(\mu(1-z))),& y &= Z(\mu).
\end{align}

We will define an operator $\cT_r$ on pairs $(U,V)$, that belong to an appropriate functional space,  as follows. Given a pair $(U,V)$, let $\lambda=\lambda(V,r)$ and $\mu=\mu(U,r)$ be the solutions  (if they exist) of the equations
\begin{align}
\label{lambda_eq} \lambda^\rho&= \Psi_{V,\lambda,r}'(0)=\lambda {r\over \rho}
{V'(\lambda r +1) \over V(\lambda r+1)}= \lambda r W'(\lambda  r),\\
\label{mu_eq}  \mu^\rho&= \Phi_{U,\mu,r}'(0)=\mu {y \over \rho^2} {U'(r+y)
\over U(r+y)} {U'(r+\mu) \over U(r+\mu) }= \mu Z'(y)Z'(\mu).  
\end{align}
Define 
\begin{equation}
\label{new_eqqq} (\tilde{U},\tilde{V}) \equiv \cT_r(U,V)= \left( \lambda^{-\rho} U \circ \Psi_{V,\lambda,r},  \mu^{-\rho} V \circ \Phi_{U,\mu,r}\right), 
\end{equation}
and 
\begin{align*}
\tilde{Z}(z)&= \prho{ \tilde{U}(r+z) \over \tilde{U}(r+\tilde{y}) }, \qquad \tilde{y}
= \prho{ \tilde{U}(r+\mu) \over \tilde{U}(r+\tilde{y}) },\\
\tilde{W}(z)&= \prho{ \tilde{V}(1+z) \over \tilde{V}(1+\lambda r) }. 
\end{align*}

In the following sections we will choose $(U,V)$ in an appropriate subset $\cS$   of a compact space of functions holomorphic in a double slit plane, and demonstrate that 
\begin{itemize}
\item[a)] for all $(U,V) \in \cS$ the solutions $(\lambda,\mu)$ of
  $(\ref{lambda_eq})-(\ref{mu_eq})$ exist and are unique for every $r$, and $\cT_r$ is a continuous operator of $\cS$ into itself;

\item[b)] $\lambda<\mu$ for sufficiently  small $r$, and $\mu>\lambda$ for sufficiently large $r$ for all $(U,V) \in \cS$; 

\item[c)] the iterates $\cT^n_r(U_0,V_0)$  converge to a fixed point $(U^*_r,V^*_r)$ of $\cT_r$ uniformly in $r$, in particular the maps $r \rightarrow \lambda^*_r=\lambda(V^*_r,r)$ and  $r \rightarrow \mu^*_r=\mu(U^*_r,r)$  are continuous for a range of positive $r$. 
\end{itemize}

These three facts will imply that there exists a value $r'$ of $r$ such that $\lambda^*_{r'}=\mu^*_{r'}$, and the pair $(U^*_{r'},V^*_{r'})$ solves $(\ref{eqq1})-(\ref{eqq2})$.

To prove a), b) and c) above, we will construct a subset of a compact space of functions $U$ and $V$ holomorphic on a double slit plane, such that the nonlinearities of $Z$ and $W$ are {\it negative} on  the real  slice of the domain  
$$ N_Z(x) \le \Sigma<0, \quad  N_W(x) \le \Gamma<0,$$
and we will demonstrate that $\tilde{Z}$ and $\tilde{W}$ for the image of $(U,V)$ under $\cT_r$ has the same bounds on the nonlinearity.

We would like to note, that assumptions on the nonlinearity seemed to be unnecessary in similar proofs of  the existence of the fixed points for the unimodal maps, for example, in the proof of existence of the Feigenbaum fixed point in  \cite{Eps1}-\cite{Eps3}.

\section{Preliminaries}

\subsection{Herglotz bounds}

We will proceed with some definitions.

The upper and the lower half planes will be denoted as
$$\fC_\pm \equiv \{z \in \fC: \pm {\Im (z)} > 0 \}.$$

Let $J=(-a,b) \subset \fR$. Given such interval $J \subset \fR$, denote
$$
\fC_J \equiv \fC_+ \cup \fC_- \cup J.   
$$ 




We will further define the space of Herglotz--Pick functions
$$\Omega(J)\equiv \left\{u: u \quad {\rm is \quad  holomorphic \quad on} \quad \field{C}_J, u(z)=\overline{u(\overline{z})}, u(0)=0 \right\}.$$
$\Omega(J)$ is a compact metric space.


 Functions in $\Omega(J)$ admit the following integral representation:
\begin{equation}\label{int_rep}
f(z)-f(z_0)= a (z-z_0) + \int d \nu(t) \left({1 \over t-z} - {1 \over t-z_0}   \right), 
\end{equation}
where $\nu$ is a measure supported in ${\fR} \setminus (-a,b)$. This integral
representation can be used to obtain the following {\it Herglotz bounds} on
$\Omega(J)$
\begin{align}
\label{first_der} { a \over x  (a+x) } &\le {f'(x) \over f(x)} \le { b \over x  (b-x) }  , \quad x \in (-a,b).
\end{align}


Notice, that the integration of the Herglotz bound $(\ref{first_der})$ gives for all $y>x >0$:
\begin{equation}\label{ratio_1}
{y (a+x) \over x  (a+y)} \le {f(y) \over f(x)} \le {y (b-x) \over x  (b-y)}.
\end{equation}

Next,  we denote by $\Omega_c(J)$ the subclass of functions  $f \in \Omega(J)$ normalized at some point $c$, $b>c>0$, as $f(c)=1$. Using the integral representation $(\ref{int_rep})$, one can  demonstrate that any $f \in  \Omega_c(J)$ satisfies the following bounds
\begin{align} 
\label{function_1}  {1  \over c} {a+c \over a+x} &\ge {f(x) \over x} \ge  {1  \over c} {b-c \over b-x}, \quad x \in (-a,c),\\
 \label{function_2}  {1  \over c} {a+c \over a+x} &\le {f(x) \over x} \le  {1  \over c} {b-c \over b-x}, \quad x \in (c,b).
\end{align}

Suppose $f \in \Omega(J)$, $J \ne \emptyset$. Then, for every $z \in J$ and every finite complex sequence $v_0,...,v_N$, one has the following relation for the derivatives of $f$
$$\sum_{j,k=0}^N{ { f^{(j+k+1)}(z) \over (j+k+1)!} v_j^* v_k }\ge 0,$$
In particular, all odd derivatives of $f \in \Omega(J)$ are non-negative on $J$, and so is the Schwarzian
\begin{equation}
\label{schwarzian}{f''' \over f'} -{3 \over 2} \left({f'' \over f' }\right)^2=\left({f'' \over f'}\right)' -{1 \over 2} \left({f'' \over f' }\right)^2 \ge 0, \quad {\rm on} \ J.
\end{equation}
In particular, the nonlinearity of a Herglotz--Pick function is increasing.

The positivity of the Schwarzian has the following consequences. Let $J=(-a,b)$ be non-empty. Denote $g=f''/f'$, suppose that $g$ is non-zero  in $[x,y] \subset J$, and integrate the inequality $g'(x) \ge g(x)^2/2$:
$${1 \over g(x)} -{1 \over g(y)} \ge {y -x  \over 2}.$$
If $g(x)>0$, then $g(y)>0$, and $g(x) \le 2 / (y-x)$, which is also true  if $g(x)<0$. At the same time, $g(y) \ge -2 / (y-x)$.  Taking the limit $y \rightarrow b$ in the first inequality, and $x  \rightarrow -a$  in the second, we get 

\begin{equation} 
\label{second_der} {-2 f'(x) \over a+x}  \le f''(x) \le {2 f'(x) \over b-x},  \quad x \in (-a,b).
\end{equation}

\subsection{Nonlinearity}

Next, assume that the nonlinearity 
$$N_f(x)={f''(x) \over f'(x)}$$
of $f$ is positive on $J$. Then, we can use the positivity of the Schwarzian derivative to obtain
$$
\left( \ln  N_f(x)  \right)'={f'''(x) \over f''(x)} -{f''(x) \over f'(x)} = {f'(x) \over f''(x)}  \left(  {f'''(x) \over f'(x)} -\left({f''(x) \over f'(x)} \right)^2 \right)  > {1 \over 2} N_f(x).
$$

Alternatively, if the nonlinearity  is negative, then
\begin{align*}
\left( \ln \left( -N_f(x) \right)  \right)'&= \left(\ln(-f''(x))-\ln(f'(x)) \right)'\\
&= {f'''(x) \over f''(x)} -{f''(x) \over f'(x)} = {f'(x) \over f''(x)}  \left(  {f'''(x) \over f'(x)} -\left({f''(x) \over f'(x)} \right)^2 \right)  < {1 \over 2} N_f(x).
\end{align*}

In either case, the solution to the the initial  value problems $\left( \ln  N_f(x)  \right)' > N_f(x)/2$, $N_f(x_0)=N_0$, or  $\left( \ln(-N_f(x) ) \right)' < N_f(x)/2$, $N_f(x_0)=N_0$,  ,  is
$$N_f(x) > {2 N_0 \over 2 - N_0 (x-x_0)}, \quad x  \ge x_0,$$
therefore,
\begin{equation}
\label{pos_1} N_f(y) \ge {2 N_f(x) \over 2 - N_f(x) (y-x)}
\end{equation}
for all $y \ge x$ in the case of a nonlinearity of a constant sign.

Furthermore,
$$\left( \ln f'(x) \right)'= N_f(x),$$
and,  under the same assumption of $N_f$ of a constant sign, the initial value problem 
$$ { f''(x) \over f'(x)} \ge {2 N_0 \over 2 - N_0 (x-x_0)}, \quad x  \ge x_0, \quad f'(x_0)=f_0, \quad f''(x_0) /f'(x_0)=N_0$$
is 
$$f'(x) \ge {4 f_0 \over \left( 2-N_0  (x-x_0) \right)^2},$$
and we get
\begin{equation}
\label{pos_2}  f'(y) \ge {4 f'(x) \over \left( 2-N_f(x)  (y-x)  \right)^2 },
\end{equation}
for all $y \ge x$.

Given a real constant $\sigma$ and  a real $c \in J=(-a,b)$, we set
$$\Omega_{\gtrless \sigma }^c(J)  \equiv  \left\{f \in \Omega(J): { \partial_x^2 p_{1 \over \rho}{\left( f(c+x) \right)} \over \partial_x p_{1 \over \rho}{\left(f(c+x)\right)}}   \gtrless \sigma, \quad x \in (-c,b-c) \right\}.$$

Notice, the set $\Omega_{\gtrless \sigma }^c(J)$ in general is not a convex subset of $\Omega(J)$.


\subsection{Schwarz Lemma}

Finally, we will mention the following easy consequence of the Schwarz Lemma  which will play an important role in our proofs below (cf \cite{Eps2}):

\begin{lemma}\label{Schwarz_lemma}
Suppose $f$ is a holomorphic map of $\fC_J$, $J=(-a,b)$, into $\fC_J'$, $J'=(-a',b')$, which fixes $0$ then 
\begin{equation}
\label{deriv} |f'(0)| \le {a' b'(a+b) \over a b (a'+b')  }.
\end{equation}
\end{lemma}

\section{Statement of results}

We will now give a more precise statement of what will be proved in the following sections.

Set 
\begin{equation}
\label{lambda-mu-plus} \lambda_+(r)=\left( r \over r+1\right)^{1 \over \rho}, \quad \mu_+(r)=\left( 1 \over (r+1)^2\right)^{1 \over \rho},
\end{equation}
and
\begin{equation}\label{Js}
J_U=\left(r\!-\!{r \over \lambda_+(r) \mu_+(r)},r\!+\!{1\over \lambda_+(r)} \right), \ J_V=\left(1\!-\!{1 \over \lambda_+(r) \sqrt{\mu_+(r)}} ,1\!+\!{r \over \mu_+(r)} \right).
\end{equation}

The intervals $J_U$ and $J_V$ serve as lower bounds on the real slices of the  maximal domains of definition of $U$ and $V$, in particular, they include the intervals $(\ref{UVdomains})$.  Notice, that the left end points of the real slices of the domains of definition of the functions $\Psi_{V,\lambda,r}$ and  $\Phi_{U,\mu,r}$ (see $(\ref{leq})$-$(\ref{mueq})$) are specified by the conditions that the arguments of the root functions $p_{1 \over \rho}$ be non-negative. This leads to our choice of the right end points in $(\ref{Js})$: $V(\lambda (r-z)+1) \ge 0 \implies \lambda (r-z)+1 \ge 0$ and $U(r-\mu(z-1)) \ge 0 \implies r-\mu(z-1) \ge 0$. The left end points of $(\ref{Js})$, are then specified by the conditions that the arguments of $V(\lambda (r-z)+1)$  and $U(r-\mu(z-1))$, appearing in the definition of $\Psi_{V,\lambda,r}$ and  $\Phi_{U,\mu,r}$ do not exceed the right end points of  $(\ref{Js})$. Also, notice that $U$ enters the definition of $\Phi_{U,\mu,r}$ in a composition with itself, which makes finding bounds for the second fixed point problem more complicated. To produce usable bounds, we had to restrict the domain of $J_V$ even further: this is the reason why $\mu_+(r)$ enters the left end point of $J_V$ as $\sqrt{\mu_+(r)}$.

\bigskip

The main results of the paper are the following two theorems.

\bigskip

\noindent {\bf Theorem A.} {\it  For any $\rho>1$,  there exist
${r_+}>{r_-}>0$, and two functions $\Sigma(r)<0$ and $\Gamma(r)<0$, continuous
on $({r_-},{r_+})$, such that:

\bigskip

\noindent i)  for every $r \in ({r_-},{r_+})$ and $(U,V) \in \Omega_{<\Sigma}^r(J_U) \times \Omega_{<\Gamma}^1(J_V)$  there is a unique solution $(\lambda,\mu) \in (0,\lambda_+(r)) \times (0,\mu_+(r))$ of the equations $(\ref{lambda_eq})$ and $(\ref{mu_eq})$, and the functions $r \mapsto \lambda(r)$, $r \mapsto \mu(r)$ are continuous on $({r_-},{r_+})$;

\bigskip

\noindent  ii)  $\cT_r$ is a well-defined, continuous operator of the subset $\Omega_{<\Sigma}^r(J_U) \times \Omega_{<\Gamma}^1(J_V)$ into itself, where $J_U$ and $J_V$ are as in $(\ref{Js})$;

\bigskip

\noindent  iii)  for any $(U_0,V_0) \in  \Omega_{<\Sigma}^r(J_U) \times \Omega_{<\Gamma}^1(J_V)$ the iterates $\cT_r^n(U_0,V_0)$ converge uniformly to a fixed point of $\cT_r$.
}

Part $i)$ of the Theorem will be mostly proved in Lemma $\ref{scalings}$, while the continuity part of the statement of Part $i)$ will be finished in the last Section $\ref{proofB}$. Part $ii)$ will be proved in Proposition $\ref{inv-nonlinearity}$. Finally, Part $iii)$ of the Theorem will be proved in Section $\ref{proofB}$.

\bigskip


Existence of a renormalization fixed point follows from the following theorem.

\bigskip

\noindent {\bf Theorem B.} {\it For every $\rho>1$ there exists $r' \in ({r_-}, {r_+})$ such that $\lambda(r)=\mu(r)$, and, therefore, the system
$$\lambda^\rho U=U \circ \Psi_{V,\lambda,r'}, \quad \lambda^\rho V=V \circ \Phi_{U,\lambda,r'}$$
has a solution $(\lambda^*,U^*_{r'},V^*_{r'}) \in (0,1) \times \Omega(J_U) \times \Omega(J_V)$.
}

Theorem B will be proved in the last Section $\ref{proofB}$.

\section{Existence of the scaling parameters for the decoupled system}\label{decoupled}

Consider functions 
$$(U,V) \in \Omega(J_U) \times \Omega(J_V),$$
and let $Z$ and $W$ be as in $(\ref{zw})$. Such $(Z,W)$  are in $\Omega(J_Z) \times \Omega(J_W)$, where
$$J_Z=\left( -r, {1\over \lambda_+}  \right), \quad J_W=\left(-1, {r  \over  \mu_+} \right).$$

We will start with a simple lemma that insures that the ``parameter'' $y$ from $(\ref{y_mu})$ is well-defined.

\begin{lemma}
For any $U\in \Omega(J_U)$, $\mu \in (0,1)$ and $\rho>1$, the equation
$$y=p_{1\over \rho}\left({ U(x+\mu) \over  U(x+y)}\right)$$
has a unique solution $y \in (\mu,1)$.

Furthermore, if $U \in \Omega^r_{<\sigma}(J_U)$ for some $\sigma<0$, then $y>y_-$, where
\begin{equation}
\label{y-} y_- \equiv {\sqrt{r^2+4 (r+\mu)} -r \over 2}>\sqrt{\mu}.
\end{equation}
\end{lemma}
\begin{proof}
Consider the function
$$f(y)=y^{\rho}U(r+y)-U(r+\mu).$$

We have
$$f(\mu)=\mu^\rho U(r+\mu)-U(r+\mu)<0, \quad f(1)=U(r+1)-U(r+\mu)>1$$
we have use that $0<\mu <1$ and $U$ is an increasing function. Therefore, $f$ has a zero in $(\mu,1)$. Furthermore, for any $y>0$
$$f'(y)=\rho y^{\rho-1} U(r+y)+y^\rho U'(r+y)>0,$$
$f$ is a monotone increasing function, and its zero in $(\mu,1)$ is unique.

To demonstrate the last claim of the Lemma, notice, that the function $Z$ is concave whenever $U \in \Omega^r_{<\sigma}(J_U)$, therefore
$$y=Z(\mu) > Z(y) {r+\mu \over r+y } ={r+\mu \over r+y}$$
(notice $Z(-r)=0$, $Z(y)=1$). The solution of this quadratic inequality yield the lower bound $(\ref{y-})$.
\end{proof}

Next, observe, that 
\begin{align*}
N_{\Phi_{U,\mu,r}}(x)&= { \Phi_{U,\mu,r}''(x)  \over \Phi_{U,\mu,r}'(x)}= -\mu  \left( {Z''(Z(\mu(1-x))) \over  Z'(Z(\mu(1-x)))}  Z'(\mu(1-x))+  {Z''(\mu(1-x)) \over  Z'(\mu(1-x))}     \right)\\
&= -\mu \left\{ N_Z(Z(\mu(1-x)))  Z'(\mu(1-x)) +N_Z(\mu(1-x))
\right\}.
\end{align*}
This implies, that whenever $U \in \Omega^r_{<\sigma}(J_U)$ for some $\sigma<0$, the function 
$$\Phi_{U,\mu,r}(x)=1-Z(Z(\mu(1-x))),$$
has \emph{positive} nonlinearity and is in $\Omega(J_\Phi)$, where
$$J_\Phi=\left(1-{y \over \mu \lambda_+} ,1+{r \over \mu} \right).$$
In particular, the analyticity of $\Phi_{U,\mu,r}$ on $\field{C}_{J_\Phi}$  follows from the fact that $Z(\mu(1-x))$ maps the interval $J_\Phi$  to $\left(Z\left(y /\lambda_+\right),0\right)$, where
 $$Z\left( {y \over \lambda_+} \right) \le {r+{y \over \lambda_+} \over r+y} \le {1 \over \lambda_+},$$
the first inequality following from concavity. Therefore, $\left(Z( y / \lambda_+),0\right)$ is contained in the domain of analyticity of $Z$. 

At the same time, the function  
$$\Psi_{V,\lambda,r}(x)=r-r W(\lambda(r-x))$$
has positive nonlinearity and is in $\Omega(J_{\Psi})$, where
$$J_\Psi=\left(r-{r \over \lambda \mu_+},r+{1\over \lambda} \right).$$ 

We are now ready to prove the following Lemma:

\begin{lemma}\label{scalings}
Let $(U,V) \in \Omega^r_{<\sigma}(J_U) \times \Omega^1_{<\gamma}(J_V)$ for some $\sigma>0$ and $\gamma>0$. 

Then, for every $r \in (0,1)$, the equations $(\ref{lambda_eq})$ and $(\ref{mu_eq})$ have a unique solution $(\lambda,\mu)$ in the set
\[
  \left(\lambda_-(r),\lambda_+(r)\right) \times \left(\mu_-(r),\mu_+(r)\right),
\]
where $\lambda_+(r)$ and $\mu_+(r)$ are as in $(\ref{lambda-mu-plus})$, and
\begin{align*}
\lambda_-(r) &= \left({r \over \rho} { 1-\sqrt{\mu_+}\lambda_+  \over  (\lambda_+ r +1  )  ( 1 +\sqrt{\mu_+} \lambda_+^2 r)  } \right)^{1 \over \rho-1} \\
\mu_-(r) &=  \left( {y_- r^2 \over \rho^2} {\left( 1  -\lambda_+  \mu_+ \right)^2 \over  ( r +1)  (r+\mu_+)  ( r  + \lambda_+ \mu_+) (r+ \lambda_+ \mu_+^2) } \right)^{1  \over \rho-1}.
\end{align*}

Furthermore, the map $(U,V) \mapsto (\lambda,\mu)$ is continuous from $\Omega^r_{<\sigma}(J_U) \times \Omega^1_{<\gamma}(J_V)$ to $\lambda_-(r),\lambda_+(r)) \times (\mu_-(r),\mu_+(r))$.

\end{lemma}
\begin{proof}
One can obtain the lower bounds on $\Psi_{V,\lambda,r}'(0)$ and
$\Phi_{U,\mu,r}'(0)$ straightforwardly from \eqref{lambda_eq}, \eqref{mu_eq}
and $(\ref{first_der})$.
\begin{align}
\label{lambda-min} \Psi_{V,\lambda,r}'(0) &\ge  {\lambda r \over \rho} { {1 \over \lambda_+ \sqrt{\mu_+}}-1  \over  ( \lambda r +1  )  ( {1 \over \lambda \sqrt{\mu_+}} +\lambda r)  },\\
\label{mu-min}  \Phi_{U,\mu,r}'(0) &\ge \mu {y \over \rho^2} { \left( {r \over \lambda_+ \mu_+} -r \right)^2\over  ( r +\mu  ) (r+y)  ( {r \over \lambda_+ \mu_+} +\mu)  ( {r \over \lambda_+ \mu_+} + y) }\\
\nonumber &\ge \mu {y_- \over \rho^2} { \left( r -r \lambda_+ \mu\_+ \right)^2\over  ( r +\mu_+  ) (r+1)  ( r  +\mu_+^2 \lambda_+)  ( r + \lambda_+ \mu_+) }.
\end{align}


On the other hand, we can use the Schwarz Lemma $\ref{Schwarz_lemma}$ to bound $\Psi_{V,\lambda,r}'(0)$ and $\Phi_{U,\mu,r}'(0)$ from above.  First, notice, that since the nonlinearities, and hence the second derivatives, of $\Psi_{V,\lambda,r}$ and $\Phi_{U,\mu,r}$ are  positive $\Psi'_{V,\lambda,r}(t) < \Psi'_{V,\lambda,r}(0)$ and  $\Phi'_{U,\mu,r}(t) < \Phi'_{U,\mu,r}(0)$ for all negative $t$ in the domain of these functions, and
\begin{align*}
\Psi_{V,\lambda,r}(-t) &>-t \Psi'_{V,\lambda,r}(0),& r-{r \over \lambda \mu_+}
&< -t <0, \\ 
\Phi_{U,\mu,r}(-t) &>-t \Phi'_{U,\mu,r}(0),& 1-{y \over \mu \lambda_+} &< -t <0. 
\end{align*}

Therefore,
\begin{align*}
\Psi_{V,\lambda,r} &: \left(-t,r+{1\over \lambda}\right) \mapsto \left(-t \Psi_{V,\lambda,r}'(0) ,r\right), \\
\Phi_{U,\mu,r}&: \left(-t,1+{r\over \mu}\right) \mapsto \left(-t \Phi_{U,\mu,r}'(0), 1- {y r \over r+\mu} \right) \subset \left(-t \Phi_{U,\mu,r}'(0), {y \over r+y} \right),
\end{align*}
where  we have used the concavity of $Z$ to get 
\begin{multline*}
Z(0) > Z(\mu) {r \over r+\mu} =y {r \over r+\mu}> Z(y) {r \over r+y}={r \over r+y} \implies \\
\Phi_{U,\mu,r}\left(1+{r \over \mu}\right) = 1-Z(0) \le 1-{r \over r+y}={y \over r+y}.
\end{multline*}

We can now use $(\ref{deriv})$ to find upper bounds on $\Psi_{U,\mu,r}'(0)$ and $\Phi_{V,\lambda,r}'(0)$:
\begin{align*}
\Psi_{V,\lambda,r}'(0) \le {t  \Psi_{V,\lambda,r}'(0) r \left( r+{1 \over \lambda} +t \right)    \over t \left(r+{1 \over \lambda}  \right) \left( t  \Psi_{V,\lambda,r}'(0) +r\right)  } &\implies \Psi_{V,\lambda,r}'(0) \le {\lambda r \over \lambda r+1} \\
\Phi_{U,\mu,r}'(0) \le {t  \Phi_{U,\mu,r}'(0) {y \over r+y}  \left( 1+{r \over
\mu} +t \right)    \over t \left(1+{r \over \mu}  \right) \left( t
\Phi_{U,\mu,r}'(0) +{y \over r+y} \right)  } &\implies \Phi_{U,\mu,r}'(0) \le {y \over r+y }{\mu  \over r+\mu}.
\end{align*}

Consider solutions of the equations $\lambda^\rho=\Psi_{V,\lambda,r}'(0)$ with
the upper and lower bounds on $\Psi_{V,\lambda,r}'(0)$ substituted for the right hand side:
\begin{align}
\nonumber \lambda^{\rho}={\lambda r \over \rho} { {1 \over \lambda_+ \sqrt{\mu_+}}-1  \over  ( \lambda r +1  )  ( {1 \over \lambda \sqrt{\mu_+}} +\lambda r)  } &\implies   \lambda \ge \lambda_-(r)=\left({r \over \rho} { 1-\sqrt{\mu_+}\lambda_+  \over  (\lambda_+ r +1  )  ( 1 +\sqrt{\mu_+} \lambda_+^2 r)  } \right)^{1 \over \rho-1} \\
\label{lambda_plus} \lambda^\rho = {\lambda r \over \lambda r+1} &\implies \lambda \le \lambda_+(r)=\left(r \over r+1 \right)^{1 \over \rho}. 
\end{align}

The function 
$$f(\lambda,r)=\lambda^{\rho-1}-{1 \over \lambda} \Psi_{V,\lambda,r}'(0)=\lambda^{\rho-1}-r  W'(\lambda r)$$
satisfies $f(\lambda_+,r) \ge 0$ and $f(\lambda_-,r) \le 0$, and 
$$\partial_\lambda f(\lambda,r)=(\rho-1) \lambda^{\rho-2} -r \partial_\lambda W'(\lambda r).$$
We have 
\begin{align*}
W'(x)&= {1 \over \rho} {1 \over V(\lambda r+1)^{1 \over \rho}} V'(x+1) V(x+1)^{{1\over \rho}-1}, \\
W'(\lambda r)&= {1 \over \rho} { V'(\lambda r +1) \over  V(\lambda r +1) }, \\
\partial_\lambda W'(\lambda r)&= {r \over \rho} \left( { V''(\lambda r +1) \over  V(\lambda r +1) } - { V'(\lambda r +1)^2 \over  V(\lambda r +1)^2 } \right),
\end{align*}
while
\begin{align*}
W''(x)&= {1 \over \rho} {1 \over V(\lambda r+1)^{1 \over \rho}} \left(  V''(x+1) V(x+1)^{{1\over \rho}-1}+\left( {1 \over \rho} -1\right) V'(x+1)^2  V(x+1)^{{1\over \rho}-2} \right), \\
W''(\lambda r)&= {1 \over \rho}\left(  {V''(x+1) \over V(x+1) } +\left( {1 \over \rho} -1 \right) { V'(x+1)^2 \over  V(x+1)^2} \right).
\end{align*}

Therefore $\partial_\lambda W'(\lambda r)< r W''(\lambda r)<0,$ since the nonlinearity, and hence the second derivative,  of $W$ is  negative. It follows that $f$ is monotone and  has a unique zero in the interval $(\lambda_-,\lambda_+)$. Continuity of $\lambda$ in $V$ follows from the fact that $f$ is continuous in $V$.

Similarly, the function
\begin{equation}\label{func_g}
g(\mu,r)=\mu^{\rho-1}-{1 \over \mu} \Phi_{U,\mu,r}'(0)=\mu^{\rho-1}-Z'(y) Z'(\mu)
\end{equation}
has a zero in the interval $(\mu_-,\mu_+)$, where
\begin{align}
 \label{mu_min} \mu_- &= \left( {y_- r^2 \over \rho^2} {\left( 1  -\lambda_+
 \mu_+ \right)^2 \over  ( r +1)  (r+\mu_+)  ( r  + \lambda_+ \mu_+) (r+
 \lambda_+ \mu_+^2) } \right)^{1  \over \rho-1}, \\
\label{mu_plus} \mu_+ &= \left(  {1 \over (r+1)^2 } \right)^{1 \over \rho}.
\end{align}

We will now show that this zero is unique. First,
\begin{equation}
\label{ZZ} \partial_\mu g(\mu,r)=(\rho-1) \mu^{\rho-2}-\left(\partial_\mu Z'(y) \right) Z'(\mu)-Z'(y) \partial_\mu Z'(\mu)
\end{equation}

Next,
\begin{multline*}
y^\rho = {U(r+\mu) \over U(r+y)} \implies \\
\partial_\mu y = { U'(r+\mu) \over \rho y^{\rho-1} U(r+y) +y^\rho U'(r+y)} \le { U'(r+\mu) \over \rho y^{\rho-1} U(r+y)} ={1\over \rho} y {U'(r+\mu) \over U(r+\mu)}
\end{multline*}

We  use this bound in an estimate on $\partial_\mu Z'(\mu)$ in the third line below:
\begin{align*}
Z'(x)&= {1 \over \rho} {1 \over U(r+y)^{1 \over \rho}} U'(r+x) U(r+x) ^{{1 \over \rho} -1},\\
Z'(\mu)&= {1 \over \rho} y {U'(r+\mu) \over U(r+\mu)},\\
Z''(x)&= {1 \over \rho} {1 \over U(r+y)^{1 \over \rho}} \left(U''(r+x) U(r+x) ^{{1 \over \rho} -1}+\left({1 \over \rho}-1  \right) U'(r+x)^2 U(r+x) ^{{1 \over \rho} -2} \right) ,\\
\partial_\mu Z'(\mu)&= {1 \over \rho} \left(\partial_\mu y \right){U'(r+\mu) \over U(r+\mu)} +{1 \over \rho} y \left({U''(r+\mu) \over U(r+\mu)} -{U'(r+\mu)^2 \over U(r+\mu)^2}  \right) \\
&\le {1\over \rho^2} y {U'(r+\mu)^2 \over U(r+\mu)^2} +{1 \over \rho} y \left({U''(r+\mu) \over U(r+\mu)} -{U'(r+\mu)^2 \over U(r+\mu)^2}  \right) \\
&\le y Z''(\mu)<0.
\end{align*}

At the same time
\begin{equation*} 
\partial_\mu Z'(y) = \left(\partial_\mu y \right){1 \over \rho} \left({U''(r+y) \over U(r+y)} -{U'(r+y)^2 \over U(r+y)^2}  \right) \le  \left(\partial_\mu y \right) Z''(y) <0.
\end{equation*}

Therefore, the right hand side of $(\ref{ZZ})$ is positive, and $g$ is a monotone increasing function. The zero of $g$ in $(\mu_-,\mu_+)$ is unique. The fact that the map $U \mapsto \mu$ is continuous follows from the continuity of the function $g$ (see $(\ref{func_g})$ ) in $U$.
\end{proof}

We will now demonstrate that the unique solutions of $(\ref{lambda_eq})$ and
$(\ref{mu_eq})$ have to satisfy $\mu>\lambda$ for sufficiently small $r$, and
$\mu<\lambda$ for sufficiently large~$r$.

\begin{lemma}\label{r-interval}
For every $\rho>1$ there exist ${r_+}={r_+}(\rho)>{r_-}={r_-}(\rho)>0$, such that the unique solution $(\lambda,\mu)$ of $(\ref{lambda_eq})$ and $(\ref{mu_eq})$ satisfy $\mu>\lambda$ for all $r<{r_-}$, and $\lambda>\mu$ for all $r>{r_+}$. 
\end{lemma}
\begin{proof}
First, we look at small $r$'s.

According to the formula $(\ref{lambda_plus})$,
\begin{equation}
\label{ineq1} \lambda^\rho r +\lambda^{\rho-1}<r,
\end{equation}
and $\lambda=O(r^{1 \over \rho})$. These two facts, in turn, imply that the first term in $(\ref{ineq1})$ is $O(r^2)$, the second --- $O(r^{{\rho-1 \over \rho}})$, and consequently, for small $r$ the inequality $(\ref{ineq1})$ becomes
$$\lambda^{\rho-1}<C r \implies \lambda=O(r^{1 \over  \rho-1}),$$
($C$ here and  below will denote an irrelevant  constant, not necessarily one and the same). At the same time, according to $(\ref{mu-min})$
\begin{align*}
\mu^{\rho-1} &\ge  C {y \over r+y } { r^2 \left( 1 - \lambda_+ \mu_+ \right)^2 \over  ( r +\mu  )  ( r  +\lambda_+ \mu_+ \mu )  ( r + \lambda_+ \mu_+ y) } \\
&\ge C {y \over r+y } { r^2 \over  ( r +\mu  )^2 ( r + O(r^{1 \over \rho-1}) ) }.
\end{align*}
Notice, that $y/(r+y)$ is an increasing function of $y$, therefore its minimum is achieved at $y_-$. For small $r$, $y_-=O(\sqrt{r+\mu})$, i.e., for small $r$, $y_-/(r+y_-)=O(1)$.
$$\mu^{\rho-1} \ge  C  { r^2 \over  ( r +\mu  )^2 ( r + O(r^{1 \over \rho-1}) ) }.$$

We consider two cases $\rho<2$ and $\rho \ge 2$. In the first case
\begin{multline*}
\mu^{\rho-1} \ge  C  { r^2 \over  ( r +\mu  )^2 ( r + O(r^{1 \over \rho-1}) ) } \ge C  { r \over  ( r +\mu  )^2} \implies \\
\mu^{{\rho-1 \over 2}}(r+\mu) \ge C r^{1\over 2} \implies  \mu \ge O(r^{1 \over \rho+1}).
\end{multline*}
in the second case
\begin{multline*}
\mu^{\rho-1} \ge  C  { r^2 \over  ( r +\mu  )^2 ( r + O(r^{1 \over \rho-1}) ) } \ge C  { r^{2-{1 \over \rho-1}} \over  ( r +\mu  )^2} \implies \\
\mu^{{\rho-1 \over 2}}(r+\mu) \ge C r^{1-{1  \over 2(\rho-1)}} \implies  \mu \ge O(r^{2 \rho -3 \over (\rho-1)(\rho+1)}).
\end{multline*}

In both cases, for sufficiently small $r$, $\mu>\lambda$.

We will now look at large $r$. First, consider $(\ref{lambda-min})$ for large $r$:
\begin{align*}
\lambda^{\rho-1} &\ge { r \over \rho} { 1 -\lambda_+ \sqrt{\mu_+}  \over  ( \lambda_+ r +1  )  ( 1 +\lambda_+ \sqrt{\mu_+} r )  } \ge   {1 \over \rho} {r \left(1+O\left({1 \over r} \right) \right) \over (\lambda_+ r +1)\left(2+O\left({1 \over r} \right) \right)}\\
&\ge  {1 \over \rho} {r \left( 1+O\left({1 \over r} \right) \right)  \over \left(  \left(1+O\left({1\over r} \right)\right)r +1\right)\left(2+O\left({1 \over r} \right)\right) }  \ge {1  \over 2 \rho } \left(1+O\left({1 \over r} \right)\right).
\end{align*}

On the other hand, $\mu_+=O\left( {1\over r^2}\right)$ (cf.~$(\ref{mu_plus})$). Therefore, for sufficiently large $r$,  $\lambda>\mu$.
\end{proof}

\section{Bounded nonlinearity}

We will now look at the images of the nonlinearities $N_Z$ and $N_W$ under the
operator $\cT_r$. Let $(U,V) \in \Omega^r_{<\sigma}(J_U) \times \Omega^1_{<\gamma}(J_V)$ for some negative $\sigma$ and $\gamma$. Then, the equations $(\ref{lambda_eq})$ and $(\ref{mu_eq})$ have a unique solution $(\lambda,\mu)$. Denote
\[
  (\tilde{U},\tilde{V}) \equiv \cT_r(U,V)= \left( \lambda^{-\rho} U \circ
  \Psi_{V,\lambda,r},  \mu^{-\rho} V \circ \Phi_{U,\mu,r}\right).
\]
Also, for brevity, denote $p_{1 \over \rho} U=U_\rho$ and $p_{1 \over \rho} V=V_\rho$, then 
\begin{align*}
N_{\tilde{Z}}(x)&= {  \left( \tilde{U}_\rho(r+x)  \right)''  \over   \left( \tilde{U}_\rho(r+x)  \right)'}=  {  \left(  U_\rho\left(\Psi_{V,\lambda,r}(r+x)\right)\right)''  \over   \left( U_\rho\left(\Psi_{V,\lambda,r}(r+x) \right) \right)' }=  {  \left(  U_\rho\left(r+\hat{\Psi}_{V,\lambda,r}(x) \right) \right)'' \over    \left(  U_\rho\left(r+\hat{\Psi}_{V,\lambda,r}(x)\right) \right)'  }  \\
&=  {  U_\rho''\left(r+\hat{\Psi}_{V,\lambda,r}(x) \right) \hat{\Psi}_{V,\lambda,r}'(x)^2+ U_\rho'\left(r+\hat{\Psi}_{V,\lambda,r}(x) \right) \hat{\Psi}_{V,\lambda,r}''(x) \over   U_\rho'\left(r+\hat{\Psi}_{V,\lambda,r}(x)\right) \hat{\Psi}_{V,\lambda,r}'(x)       }  \\
&=  N_Z\left( \hat{\Psi}_{V,\lambda,r}(x) \right) \hat{\Psi}_{V,\lambda,r}'(x)+ {\hat{\Psi}_{V,\lambda,r}''(x) \over \hat{\Psi}_{V,\lambda,r}'(x)}  \\
&=  N_Z\left( \hat{\Psi}_{V,\lambda,r}(x) \right) \hat{\Psi}_{V,\lambda,r}'(x)+ N_{\Psi_{V,\lambda,r}}(r+x).
\end{align*}
where $\hat{\Psi}_{V,\lambda,r}(x)={\Psi}_{V,\lambda,r}(r+x)-r.$

\begin{align*}
N_{\tilde{W}}(x)&= {  \left( \tilde{V}_\rho(1+x)   \right)''  \over   \left(  \tilde{V}_\rho(1+x)   \right)'}=  {  \left( V_\rho\left(\Phi_{U,\mu,r}(1+x)\right) \right)''  \over   \left( V_\rho\left(\Phi_{U,\mu,r}(r+x) \right) \right)' }=  {  \left( V_\rho \left(r+\hat{\Phi}_{U,\mu,r}(x)\right) \right)'' \over    \left(  V_\rho \left(r+\hat{\Phi}_{U,\mu,r}(x)\right)  \right)'      } \\
&= {  V_\rho ''\left(r+\hat{\Phi}_{U,\mu,r}(x)\right) \hat{\Phi}_{U,\mu,r}'(x)^2 +  V_\rho '\left(r+\hat{\Phi}_{U,\mu,r}(x)\right) \hat{\Phi}_{U,\mu,r}''(x)  \over    V_\rho'\left(r+\hat{\Phi}_{U,\mu,r}(x)\right) \hat{\Phi}_{U,\mu,r}'(x)  } \\
&= N_W \left( \hat{\Phi}_{U,\mu,r}(x) \right) \hat{\Phi}_{U,\mu,r}'(x)+ {\hat{\Phi}_{U,\mu,r}''(x)  \over \hat{\Phi}_{U,\mu,r}'(x)}\\
&= N_W \left( \hat{\Phi}_{U,\mu,r}(x) \right) \hat{\Phi}_{U,\mu,r}'(x)+ N_{\Phi_{U,\mu,r}}(1+x).
\end{align*}
where $\hat{\Phi}_{U,\mu,r}(x)={\Phi}_{U,\mu,r}(1+x)-1.$

We will also require the following relation between the nonlinearity of $\Phi_{U,\mu,r}$ and~$Z$:
\begin{equation}
\label{nonl-relation}   N_{\Phi_{U,\mu,r}}(x)=-\mu \left\{ N_Z(Z(\mu(1-x))) Z'(\mu(1-x))+N_Z(\mu(1-x)) \right\}. 
\end{equation}

Recall, that for Herglotz--Pick functions
$$N'_f(x)={f'''(x)  \over f'(x)}-\left( {f''(x)  \over f'(x)}\right)^2  \ge {1 \over 2}  N_f(x)^2,$$
which follows from the positivity of the Schwarzian derivative. Therefore, the nonlinearity of these functions is monotone increasing.

\begin{prop}\label{inv-nonlinearity}
There exist functions $\Sigma(r)<0$ and $\Gamma(r)<0$, continuous in $r$, such that $\cT_r$ is a continuous operator of the set $\Omega^r_{<\Sigma(r)}(J_U) \times \Omega^1_{<\Gamma(r)}(J_V)$ into itself.
\end{prop}
\begin{proof}

For a fixed $r$, suppose the nonlinearity of $Z$ is bounded by some negative $\sigma$  on all of $J_Z=(-r,1/\lambda_+)$, while that of $W$ is bounded by some $\gamma$ on $J_W=(-1,r/\mu_+)$:
\begin{equation}
\label{nonl-bounds} N_Z \left( {1 \over \lambda_+} \right) \le \sigma  <0,  \quad   N_W \left( {r \over \mu_+} \right) \le \gamma <0,
\end{equation}
Notice, due to the bounds $(\ref{second_der})$,
\begin{equation}
\label{eps-delt_-} \sigma \ge -{2 \over  r +{1\over \lambda_+}} \equiv \sigma_-, \quad \gamma \ge -{2 \over  1 +{r\over \mu_+}}\equiv \gamma_-.
\end{equation}

Let $(\lambda,\mu)$ be the unique solution of $(\ref{lambda_eq})$ and $(\ref{mu_eq})$. Below, we will assume a certain form of the bounds $\sigma$ and $\gamma$, and we will show that $N_{\tilde{Z}}(1/ \lambda_+)$ and $N_{\tilde{W}}(r /\mu_+)$ satisfy bounds of the same form.  In fact, we will estimate the maximum of $N_{\tilde{Z}}(1/l)$  for any $\lambda<l<\lambda_+$ and the maximum of $N_{\tilde{W}}(r/m)$ for any  $\mu<m<\mu_+$, and use the fact that $N_{\tilde{Z}}(1/l)>N_{\tilde{Z}}(1 /\lambda_+)$ and  $N_{\tilde{W}}(r/m)>N_{\tilde{W}}(r / \mu_+)$. The exact reason for why the nonlinearities are estimated at points $1/l>1/\lambda_+$ and $r/ m > r / \mu_+$ will be given at the end of Step 1).

\bigskip

\noindent \textit{Step 1).} We start with  $N_{\tilde{Z}}(1/l)$: 
$$N_{\tilde{Z}} \left({1 \over l }  \right)=N_Z\left(  \hat{\Psi}_{V,\lambda,r}\left( {1 \over l} \right)  \right) \hat{\Psi}_{V,\lambda,r}'\left({1 \over l} \right)+ N_{\Psi_{V,\lambda,r}}\left(r+{1 \over l} \right).$$

Since $\hat{\Psi}_{V,\lambda,r}$ is an  increasing function,
$$\hat{\Psi}_{V,\lambda,r}\left( {1 \over l} \right) \le \hat{\Psi}_{V,\lambda,r}\left( {1 \over \lambda} \right) = \Psi_{V,\lambda,r}\left(r+ {1 \over \lambda} \right)-r=0,$$ 
and 
$$N_{\tilde{Z}} \left({1 \over l }  \right) \le N_Z (0) \hat{\Psi}_{V,\lambda,r}'\left({1 \over l} \right)+ N_{\Psi_{V,\lambda,r}}\left(r+{1 \over l} \right).$$

Notice, that 
\begin{equation}
\nonumber  N_{\Psi_{V,\lambda,r}}(x)=-\lambda N_{W}(\lambda(r-x)),
\end{equation}
therefore,
\begin{equation}
\label{bound5}
N_{\Psi_{V,\lambda,r}}\left(r+{1 \over l} \right)=-\lambda  N_{W}\left(-{\lambda \over l}  \right), \quad N_{\Psi_{V,\lambda,r}}(0)=-\lambda N_{W}(\lambda r).
\end{equation}

The estimate $(\ref{pos_1})$ can be used to  bound $N_Z(0)$ and  $N_{W}(\lambda r)$ from  above:
\begin{equation}
\label{bound2} N_Z(0)  \le {2 N_Z\left( {1 \over \lambda_+} \right)  \over 2 + N_Z\left( {1 \over \lambda_+} \right)  {1 \over \lambda_+} }, \quad  N_W(\lambda r)  \le {2 N_W\left( {r  \over \mu_+} \right)  \over 2 + N_W\left( {r \over \mu_+} \right)  {\left( {r \over \mu_+}-\lambda r\right) }}.
\end{equation}

We also use the bound $(\ref{pos_2})$ to estimate $\hat{\Psi}_{V,\lambda,r}'\left({1 / l} \right)$ from below.
\begin{equation}
\label{bound4} \hat{\Psi}_{V,\lambda,r}'\left({1 \over l} \right)= \Psi_{V,\lambda,r}'\left(r+{1 \over l} \right) \ge {4 \lambda^\rho \over \left(2-N_{\Psi_{V,\lambda,r}}(0) \left( r+ {1\over l}\right) \right)^2 }.
\end{equation}

We collect the estimates $(\ref{bound2})$, $(\ref{bound4})$ and  $(\ref{bound5})$, and use $(\ref{second_der})$ on $N_W(-\lambda/l)$  to get
\begin{align}
\nonumber N_{\tilde{Z}} \left({1 \over l }  \right) &\le N_Z\left( {1 \over \lambda_+} \right)  {8 \lambda^\rho \over \left( 2 + N_Z\left( {1 \over \lambda_+} \right)  {1 \over \lambda_+}   \right) \left( 2-N_{\Psi_{V,\lambda,r}}(0) \left( r+ {1\over l}\right)\right)^2    }  \\
\nonumber &\quad -\lambda  N_{W}\left(-{\lambda \over l} \right)\\
\begin{split}
  &\le N_Z\left( {1 \over \lambda_+} \right)  {8 \lambda^\rho \over \left( 2 + N_Z\left( {1 \over \lambda_+} \right)  {1 \over \lambda_+}   \right) \left( 2+\lambda N_{W}(\lambda r) \left( r+ {1\over l}\right)\right)^2    }  \\
\label{bound1} &\quad+\lambda  {2 \over 1-{\lambda \over l}}.
\end{split}
\end{align}

 To demonstrate that there are  $\sigma$ and $\gamma$ such that the
 nonlinearities $N_{\tilde{Z}}$ and $N_{\tilde{W}}$ satisfy the bounds in $(\ref{nonl-bounds})$,  it is sufficient to come up with  a choice of these constants so that the upper bound $\tilde{\sigma}$ on $N_{\tilde{Z}} \left({1 / l}  \right)$ is less than $\sigma$:
$$\sigma>\tilde{\sigma}$$
where 
\begin{align*}
\tilde{\sigma} &= \sigma  {8 \lambda^\rho \over f} +\lambda  {2 \over 1-{\lambda \over l}},  \quad  f(\sigma,\gamma,l)= \left( 2 + {\sigma \over \lambda_+}   \right) \left( 2+\lambda a(\gamma) \left( r+ {1\over l}\right)\right)^2, \quad ({\rm cf}. \ (\ref{bound1})),\\
a(\gamma)&= {2 \gamma \over 2 + \gamma  {\left( {r \over
\mu_+}-\lambda r\right) } } \quad (\text{cf.\ the second equation
of~\ref{bound2}}).
\end{align*}

Recall the definition $(\ref{eps-delt_-})$ of $\sigma_-$ and $\gamma_-$, and set
$$\Sigma=-{2-(m-\mu)  \over r+ {1 \over \lambda_+}}, \quad \Gamma=-{2-(l-\lambda)  \over 1+ {r \over \mu_+}},$$
Notice,  that
\begin{align*}
f(\Sigma,\Gamma,l)&= {4 (r l +\mu_+ (2+\lambda))(2r \lambda_++(m-\mu)) \over \left( -2\mu_+ (1+\lambda r) +r  (\lambda-l)(1-\lambda \mu_+) \right)^2 (\lambda_+ r+1 ) l^2  } (l-\lambda)^2 \\
&= O(1) (l-\lambda)^2,
\end{align*}
where $O(1)$ is a positive function of $r$, $\lambda$, $l$, $\lambda_+$ and  $\mu_+$ of order $0$ in $(l-\lambda)$. Consider the function
$$g(\sigma,\gamma,l)=(\tilde{\sigma} -\sigma) f(\sigma,\gamma,l) =\sigma  8 \lambda^\rho  +{2  \lambda l \over l-\lambda} f(\sigma,\gamma,l) -\sigma f(\sigma,\gamma,l).$$

We have
\begin{equation}
\label {g} g(\Sigma,\Gamma,l)=-8\lambda^\rho \cdot {2-(m-\mu)  \over r+ {1 \over \lambda_+}}+O(1) (l-\lambda)+O(1) (l-\lambda)^2.
\end{equation}

For any  $\lambda_-<\lambda  < \lambda_+$, we can choose $l$, sufficiently  close to $\lambda$, so that $g(\Sigma,\Gamma,l) \le 0$.  Therefore, for such $l$,
$$N_{\tilde{Z}}\left( 1 \over \lambda_+\right) -\sigma \le {g(\Sigma,\Gamma,l) \over f(\Sigma,\Gamma,l)} \le 0,$$
as required. 

It is clear at this point why we chose to estimate $N_{\tilde{Z}}$ at $1/l$ and not at $1 / \lambda_+$:  had $l$ been chosen equal to $\lambda_+$ from the beginning, one would have to deal with the positive  second and third  terms in $(\ref{g})$. Specifically, one would have to show that $g$, a complicated function of $r$ and  $\lambda$, is negative for a wide range of $r$'s. We have circumvented this problem by estimating the nonlinearity at a point $1/l>1/\lambda_+$  (recall, $N_{\tilde{Z}}(1/l)>N_{\tilde{Z}}(1/\lambda_+)$), and using the fact that $l$ can be freely chosen to be close to  $\lambda$, the solution of $(\ref{lambda_eq})$, so that the second and the third terms in $(\ref{g})$ are small in the absolute value compared to the first one. 

The solution $(\lambda,\mu)$ clearly depends on $(U,V)$, and, seemingly, so do $\Sigma$ and $\Gamma$. However, we can set
$$l=\lambda+\delta(r), \quad m=\mu+\epsilon(r),$$
and choose $\delta(r)$ and $\epsilon(r)$ to be  continuous positive functions of $r$ {\it only}, sufficiently (but not necessarily infinitesimally) small, so that $g<0$. Then
$$\Sigma=-{2-\delta(r)  \over r+ {1 \over \lambda_+(r)}}, \quad \Gamma=-{2-\epsilon(r)  \over 1+ {r \over \mu_+(r)}}$$
are continuous functions of $r$ {\it only}.

\bigskip

\noindent \textit{Step 2).} We will now consider the maximum of $N_{\tilde{W}}(r/m)$, in a similar way. For any $ \mu < m < \mu_+$
\begin{equation}
\label{newNZ} N_{\tilde{W}} \left({r \over m }  \right)=N_W\left(  \hat{\Phi}_{U,\mu,r}\left( {r \over m} \right)  \right) \hat{\Phi}_{U,\mu,r}'\left({r \over m} \right)+ N_{\Phi_{U,\mu,r}}\left(1+{r \over m} \right).
\end{equation}

First, by concavity of $Z$,  
$${Z(0) \over r} \ge {Z(y) \over r+y}={1 \over r+y},$$
therefore,
$$\hat{\Phi}_{U,\mu,r}\left( {r \over m} \right)=-Z\left(Z\left(-{\mu \over m } r \right) \right)  \le -Z(0) \le -{r \over r+y} \le -{r \over r+1},$$
and
$$N_{\tilde{W}} \left({r \over m }  \right) \le N_W\left( -{r \over r+1} \right)  \hat{\Phi}_{U,\mu,r}'\left({r \over m} \right)+ N_{\Phi_{U,\mu,r}}\left(1+{r \over m} \right).$$

We use the bound $(\ref{pos_2})$ to estimate $\hat{\Phi}_{U,\mu,r}'\left({r / m} \right)$ from below.
\begin{equation}
\label{est4}
 \hat{\Phi}_{U,\mu,r}'\left({r \over m} \right)= \Phi_{U,\mu,r}'\left(1+{r \over m} \right) \ge {4 \mu^\rho \over \left( 2-N_{\Phi_{U,\mu,r}}(0) \left( 1+ {r\over m}\right)  \right)^2}.
\end{equation}

The bound $(\ref{pos_1})$ can be used to  bound $N_W(-r/(r+1))$ from  above:
\begin{equation}
\label{est2} N_W \left(-{r \over r+1}\right)  \le {2 N_W\left( {r \over \mu_+} \right)  \over 2 + N_W\left( {r \over \mu_+} \right)  \left( {r \over r+1} +{r \over \mu_+}\right) }.
\end{equation}

We substitute   the  estimates $(\ref{est2})$ and $(\ref{est4})$ in $(\ref{newNZ})$, together with the estimate $(\ref{second_der})$ for  $N_{\Phi_{U,\mu,r}}(1+r/m)$,  to get
\begin{align}
\nonumber N_{\tilde{W}} \left({r \over m}  \right) &\le  N_W\left( {r \over \mu_+} \right)  {8 \mu^\rho \over \left( 2 + N_W\left( {r \over \mu_+} \right) \left( {r \over r+1}+  {r \over \mu_+}  \right)  \right) \left( 2-N_{\Phi_{U,\mu,r}}(0) \left( 1+ {r\over m}\right)\right)^2    } \\
\nonumber &\quad +N_{\Phi_{U,\mu,r}}\left(1+{r  \over m }\right) \\
\begin{split}
&\le  N_W\left( {r \over \mu_+} \right)  {8 \mu^\rho \over \left( 2 + N_W\left( {r \over \mu_+} \right) \left( {r \over r+1}+  {r \over \mu_+}  \right)  \right) \left( 2-N_{\Phi_{U,\mu,r}}(0) \left( 1+ {r\over m}\right)\right)^2    } \\
\label{est7} &\quad +{2 \over {r \over \mu} -{r \over m}}.
\end{split}
\end{align}

Next, by the relation $(\ref{nonl-relation})$
\begin{equation}
\label{est8} N_{\Phi_{U,\mu,r}}(0)=-\mu \left\{N_Z(y) Z'(\mu) +N_Z(\mu)  \right\} \ge -\mu N_Z(\mu).
\end{equation}

Again, we use the bound $(\ref{pos_1})$ to estimate $N_Z(\mu)$ from above.
\begin{equation}
  \label{est33}   N_Z (\mu)  \le {2 N_Z\left( {1 \over \lambda_+} \right)  \over 2 + N_Z\left( {1 \over \lambda_+} \right)  \left( {1 \over \lambda_+} -\mu \right) }. 
\end{equation}



It is sufficient to come up with  a choice of  $\sigma$ and $\gamma$ so that the upper bound on $N_{\tilde{W}} \left({r / m }  \right)$ is less than $\gamma$:
$$\gamma \ge \tilde{\gamma},$$
where 
\begin{align*}
\tilde{\gamma} &= \gamma  {8 \mu^\rho \over h} +{2 \over {r \over \mu} -{r \over m}}, \quad h(\sigma, \gamma,m)=\left( 2 + \gamma \left( {r \over r+1} +  { r \over \mu_+}   \right) \right) \left( 2+\mu b(\sigma) \left( 1+ {r\over m }\right)\right)^2,\\
b&= {2 \sigma \over 2 + \sigma  \left( {1 \over \lambda_+} -\mu \right) } \quad
({\rm cf}. \ (\ref{est8}), (\ref{est33})).
\end{align*}

Let be $\Sigma$ and $\Gamma$ be as in $(\ref{g})$. Notice, that 
\begin{align*}
h(\Sigma,\Gamma,m)&= { 4 ( 2 \lambda_+ r  + (m+\mu \lambda_+ r))^2  ( 2 \mu_+ +r (l-\lambda)(\mu_++r+1)) \over (2 \lambda_+(r+\mu) +(m-\mu) (1-\mu\lambda_+))^2 (r+1)(r+\mu_+)m^2 }(m-\mu)^2   \\
&= O(1) (m-\mu)^2.
\end{align*}

Consider the function
$$z(\sigma,\gamma,m)=(\tilde{\gamma}-\gamma) h(\sigma,\gamma,m)=\gamma 8 \mu^\rho+{2 \mu m \over r}  {h(\sigma,\gamma,m) \over (m-\mu)}-\gamma h(\sigma,\gamma,m).$$

We have 
$$z(\Sigma,\Gamma,m)=-8 \mu^\rho {2-(l-\lambda) \over r +{1 \over \lambda_+}} + O(1) (m-\mu) +O(1) (m-\mu)^2.$$

As we have already discussed, for any $\mu_- <\mu < \mu_+$ we can choose 
$$m=\mu+\epsilon(r)$$
where $\epsilon(r)$ is positive, sufficiently small and  continuous, so that $z(\Sigma,\Gamma,m) \le 0$. Therefore, 
$$N_{\tilde{W}}\left( r \over m\right) -\gamma \le {z(\Sigma,\Gamma,m) \over  h(\Sigma,\Gamma,m)} \le 0$$
as needed.

\bigskip

\noindent \textit{Step 3).} We shall now prove the claim about the continuity of the operator $\cT_r$. Recall, that according to Lemma $(\ref{scalings})$, for every fixed $r$, the map  $(U,V) \mapsto (\lambda(V),\mu(U))$, is continuous from $\Omega^r_{<\Sigma(r)}(J_U) \times \Omega^1_{<\Gamma(r)}(J_V)$ to $(\lambda_-(r),\lambda_+(r)) \times (\mu_-(r),\mu_+(r))$. This, together with the continuity of $\Psi_{V,\lambda,r}$ in $U$ and $\lambda$, and $\Phi_{U,\mu,r}$ in $U$ and $\mu$, implies that the map
$$(U,V) \mapsto \left( \lambda^{-\rho}(V) U \circ \Psi_{V,\lambda(V),r}, \mu^{-\rho}(U) V \circ \Phi_{U,\mu(U),r} \right)$$
is continuous from $\Omega^r_{<\Sigma(r)}(J_U) \times \Omega^1_{<\Gamma(r)}(J_V)$ to itself.

\end{proof}

\begin{figure}

 \begin{center}
\begin{tabular}{c c}
 {\includegraphics[height=55mm,width=55mm,angle=-90]{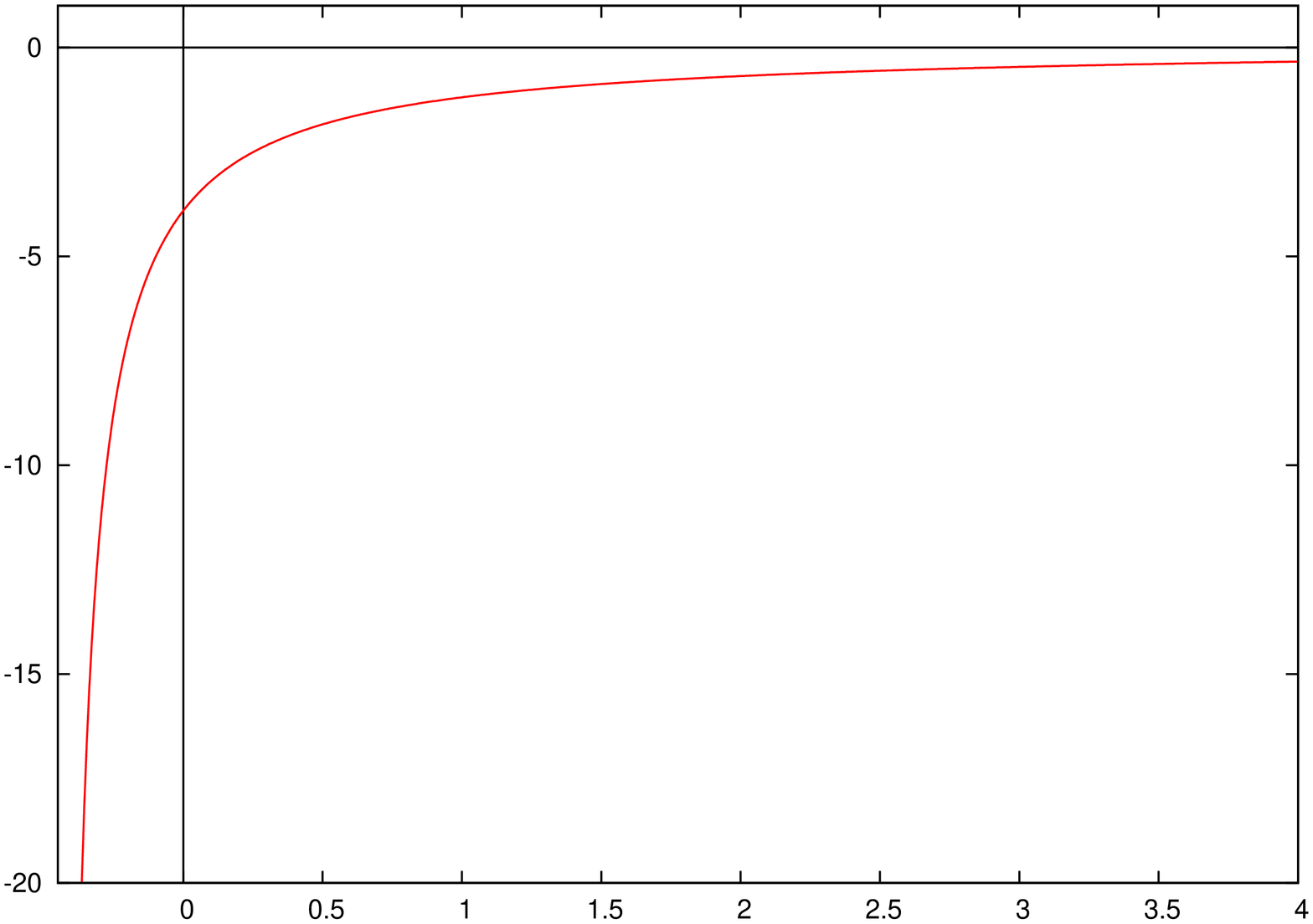}} & {!}{\includegraphics[height=55mm,width=55mm,angle=-90]{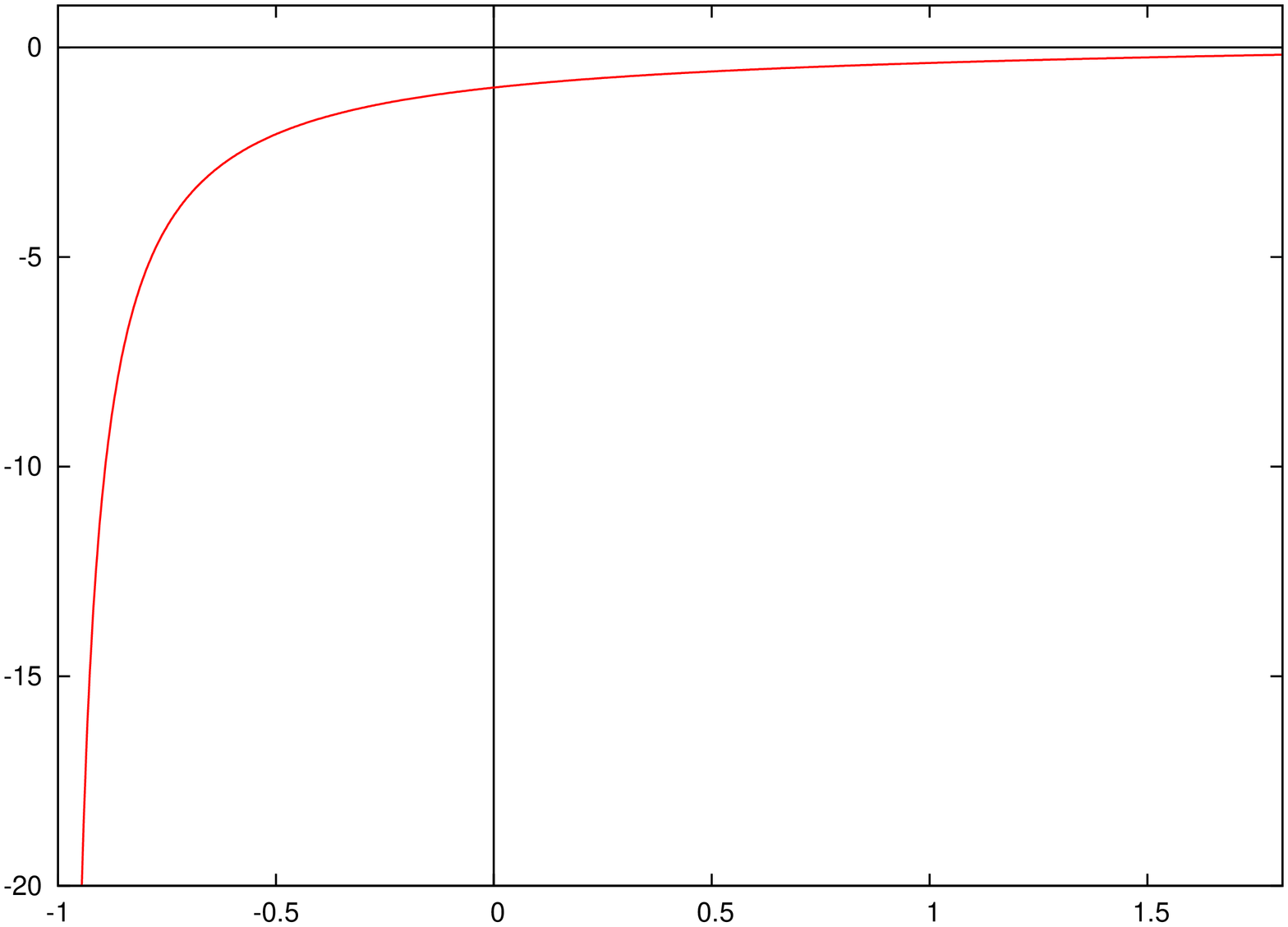}} 
\end{tabular} 
\caption{ \it Nonlinearities $N_Z$ (a)) and $N_W$ (b)) for the fixed point ($\rho=2$) computed in \cite{Win1}.}
\end{center}
\end{figure}

\section{Existence of a renormalization fixed point}\label{proofB}

In this section we prove the continuity statement of Part $i)$ of Theorem A. We also prove all of Part $iii)$ of Theorem A, and Theorem B.

\begin{proof}
In Lemmas $\ref{scalings}$ and  $\ref{r-interval}$, and in Prop.
$\ref{inv-nonlinearity}$  we have shown that there exists an interval
$({r_-},{r_+})$, $0<{r_-}<{r_+}$, of parameter values $r$, and, for every $r
\in ({r_-},{r_+})$, intervals $J_U$ and $J_V$ and bounds $\Sigma(r)$ and
$\Gamma(r)$, continuous in $r$, such that the operator $\cT_r$ maps the
relatively compact set $\Omega^r_{<\Sigma(r)}(J_U) \times \Omega^1_{<\Gamma(r)}(J_V)$ into itself, and, furthermore,  $\lambda^\rho(V,r)$ and $\mu^\rho(U,r)$ are contained in some subinterval $[ \Delta,(1- \Delta)^2] \subset (0,1)$, $\mu_-({r_-})>\lambda_+(r_+)$ and  $\lambda_-(V,r_+)>\mu_+(r_+)$.

Now, consider a sequence $(U_n,V_n) \equiv \cT_r^n(U_0,V_0)$ (cf.~$(\ref{new_eqqq})$) with $(U_0,V_0) \in  \Omega_{<\Sigma}^r(J_U) \times \Omega^1_{<\Gamma}(J_V)$. Notice, since $\Sigma(r)$ and $\Gamma(r)$ are continuous functions of $r$, and since the dependence of $J_U$ and $J_V$ on $r$ is also continuous, such $(U_0,V_0)$ can be chosen in such a way that the the map $r \mapsto (r,U_0,V_0)$ is continuous from $({r_-},r_+)$ to $\sqcup_{r \in ({r_-},r_+)} \Omega^r_{<\Sigma(r)}(J_U) \times \Omega^1_{<\Gamma(r)}(J_V)$ (here $\sqcup$ stands for a disjoint union).

 Since the set $\Omega_{<\Sigma}^r(J_U) \times \Omega^1_{<\Gamma}(J_V)$ is invariant under $\cT_r$, the scalings  $\lambda_n^\rho=\Psi_{V_n,\lambda_n,r}'(0)$ and $\mu_n^\rho=\Phi_{U_n,\mu_n,r}'(0)$ satisfy the bounds from Lemma $\ref{scalings}$, and are in $[\Delta,(1-\Delta)^2] \subset (0,1)$. Consider
\begin{align*}
U_n&= {1 \over \lambda_{n-1}^\rho  \ldots \lambda_1^\rho \lambda_0^\rho }
\Psi_{V_{n-1},\lambda_{n-1},r}  \circ \ldots \circ \Psi_{V_1,\lambda_1,r} \circ
\Psi_{V_0,\lambda_0,r},& \lambda_k^\rho &=\Psi_{V_k,\lambda_k,r}'(0),\\
V_n&= {1 \over \mu_{n-1}^\rho \ldots \mu_1^\rho \mu_0^\rho }
\Phi_{U_{n-1},\mu_{n-1},r} \circ \ldots \circ \Phi_{U_1,\mu_1,r} \circ
\Phi_{U_0,\mu_0,r},& \mu_k^\rho &=\Phi_{U_k,\mu_k,r}'(0).
\end{align*}

For each $k$, there exist  neighborhoods $D_U^k$ and $D_V^k$ of zero, such that 
$$\left|\Psi_{V_k,\lambda_k,r}(z)\right| <A_k |z|, \quad z \in D_U^k, \quad {\rm and} \quad \left|\Phi_{U_k,\mu_k,r}(z)\right| <B_k |z|, \quad z \in D_V^k,$$
for some $A_k < 1$ and $B_k <1$, such that $A_k^2<\lambda_k^\rho<A_k$ and $B_k^2<\mu_k^\rho<B_k$. Specifically, one can choose $A_k=\lambda^\rho_k (1+\Delta)$ and $B_k=\mu^\rho_k (1+\Delta)$.  Notice, that since $A_k \ge \lambda_k^\rho + \Delta^2>\Delta+\Delta^2$ (and similarly for $B_k$), the domains $D_U^k$ and $D_V^k$ do not shrink to zero, that is, there exists $s>0$, such that $\field{D}_s(0) \subset D_U^k$ and  $\field{D}_s(0) \subset D_V^k$ for all $k \ge 0$ (here, $\field{D}_s(0)$ denotes a  disk around $0$ of radius $s$ in $\fC$).  Thus, for any $z \in \field{D}_s(0)$
\begin{align*}
\left|\Psi_{V_{n-1},\lambda_{n-1},r}  \circ \ldots \circ \Psi_{V_1,\lambda_1,r} \circ \Psi_{V_0,\lambda_0,r}(z)\right| &\le \prod_{k=0}^{n-1} A_k s,\\
\left|\Phi_{U_{n-1},\mu_{n-1},r}  \circ \ldots \circ \Phi_{U_1,\mu_1,r} \circ \Phi_{U_0,\mu_0,r}(z)\right| &\le \prod_{k=0}^{n-1} B_k s.
\end{align*}

Furthermore, for all $z \in \field{D}_s(0)$, $|\Psi_{V_k,\lambda_k,r}(z)-\lambda^\rho_k z| \le K |z^2|$ for some constant $K$, therefore 
\begin{align*}
\left|\Psi_{V_n,\lambda_n,r}  \circ \ldots \circ \Psi_{V_0,\lambda_0,r}(z) -\lambda^\rho_n \Psi_{V_{n-1},\lambda_{n-1},r}  \circ \ldots \circ \Psi_{V_0,\lambda_0,r}(z)\right| &\le K  \left( \prod_{k=0}^{n-1} A_k s \right)^2,\\
\left|\Phi_{U_n,\mu_n,r}  \circ \ldots \circ \Phi_{U_0,\mu_0,r}(z) -\mu^\rho_n \Phi_{U_{n-1},\mu_{n-1},r}  \circ \ldots \circ \Phi_{U_0,\mu_0,r}(z)\right| &\le K  \left( \prod_{k=0}^{n-1} B_k s \right)^2,
\end{align*}
while, with our choice of $A_k$ and $B_k$,
\begin{align}
\label{uniformU} \left|U_{n+1}(z) -U_{n}(z)\right| &\le {K s^2 \over \lambda_n^\rho}  \prod_{k=0}^{n-1} {A_k^2 \over  \lambda_k^\rho}  \le {K s^2 \over \Delta}  (1-\Delta^2)^{2n},\\
\label{uniformV} \left|V_{n+1}(z) -V_{n}(z)\right| &\le {K s^2 \over \mu_n^\rho}  \prod_{k=0}^{n-1} {B_k^2 \over \mu_k^\rho} \le {K s^2 \over \Delta}  (1-\Delta^2)^{2n}.
\end{align}

We therefore obtain that for every $r \in ({r_-},r_+)$ the sequences $(U_n,V_n)$ converges uniformly on $\field{D}_s(0)$, and, in fact, as $(\ref{uniformU})$ and $(\ref{uniformV})$ show, the rate of convergence is independent of $r$. Since every $(U_n,V_n) \in \Omega(J_U) \times \Omega(J_V)$ this convergence is uniform on every compact subset of $\field{C}_{J_U} \times \field{C}_{J_V}$.  The limit of this sequence, $(U^*_r,V^*_r)$, satisfies 
\begin{align*}
\lambda^\rho(V^*_r,r) \ U^*_r&= U^*_r \circ \Psi_{V^*_r,\lambda(V^*_r,r),r},\\
\mu^\rho(U^*_r,r) \ V^*_r    &= V^*_r \circ \Phi_{U^*_r,\mu(U^*_r,r),r},
\end{align*}
on any compact subset of   $\field{C}_{J_U} \times \field{C}_{J_V}$, and, by extension, on all of   $\field{C}_{J_U} \times \field{C}_{J_V}$.  If $U_0$ and $V_0$ are univalent in $\field{C}_{J_U}$ and $\field{C}_{J_V}$, respectively, then so are $U_n$ and $V_n$, and, by Hurwitz convergence theorem, so are $U_r^*$ and $V_r^*$.

We will proceed to demonstrate by induction that the map $r \mapsto (r,U_{n+1},V_{n+1},\lambda_{n},\mu_{n})$ is continuous from $({r_-},r_+)$ to  $\sqcup_{r \in ({r_-},r_+)} \Omega^r_{<\Sigma(r)}(J_U) \times \Omega^1_{<\Gamma(r)}(J_V) \times [\Delta,(1-\Delta)^2]^2$.

Consider the functions
$$f(\lambda;(r,V))=\lambda^{\rho-1}-{r \over \rho}{V'(\lambda r+1) \over V(\lambda r+1)}, \quad f_0(\lambda,r) \equiv f(\lambda;(r,V_0)).$$
Since both $\lambda$ and $\mu$ are contained in $[\Delta,(1-\Delta)^2]$, the points $\lambda r+1$, $r+\mu$ and $r+y$ (recall, $y<1$) are always contained compactly in  $J_V$ and $J_U$ respectively, and the map
$$(\lambda;(r,V)) \mapsto f(\lambda;(r,V))$$
is clearly continuous from $[\Delta,(1-\Delta)^2] \times \sqcup_{r \in ({r_-},r_+)}\Omega^1_{<\Gamma(r)}(J_V)$ to $C^0([\Delta,(1-\Delta)^2] \times \sqcup_{r \in ({r_-},r_+)} \Omega^1_{<\Gamma(r)}(J_V),\field{R})$.

Recall, that according to the Lemma $(\ref{scalings})$, for each $r\in({r_-},r_+)$ the function $f_0$ has a single zero in $(\lambda_-(r),\lambda_+(r))$. Since $(U_0,V_0)$ have been chosen to  be  continuous functions of  $r$, and since the map $z \mapsto V_0(z)$ is  continuously differentiable in $J_{V} \ni \lambda r+1$, the map $(\lambda,r) \mapsto f_0(\lambda,r)$ is continuous, and so is the unique zero  of $f_0(\cdot,r)$: the map $r \mapsto \lambda_0(r)$ is continuous from $({r_-},r_+)$ to $[\Delta,(1-\Delta)^2]$. One can argue in a similar way, that the map $r \mapsto \mu_0(r)$ is continuous from $({r_-},r_+)$ to $[\Delta,(1-\Delta)^2]$ as well.

 This, together with the  continuity of $\Psi$ in $V$, $\lambda$ and $r$, and $\Phi$ in $U$, $\mu$ and $r$, implies that  the map 
$$r \mapsto (r,U_1,V_1)=\left(r, \lambda^{-\rho}_0 U_0 \circ \Psi_{V_0,\lambda_0,r}, \mu^{-\rho}_0 V_0 \circ \Phi_{U_0,\mu_0,r}\right)$$ 
is continuous from  $({r_-},r_+)$ to $\sqcup_{r \in ({r_-},r_+)} \Omega^r_{<\Sigma(r)}(J_U) \times \Omega^1_{<\Gamma(r)}(J_V)$. 

Next, assume, that $r \mapsto (r,U_{n},V_{n},\lambda_{n-1},\mu_{n-1})$ is continuous. Then, one can  argue identically to the case $n=1$ above (substituting $(U_{n},V_n)$ for  $(U_{0},V_0)$) that 
$$r \mapsto (r,U_{n+1},V_{n+1},\lambda_{n},\mu_{n})$$
is continuous from  $({r_-},r_+)$ to $\sqcup_{r \in ({r_-},r_+)} \Omega^r_{<\Sigma(r)}(J_U) \times \Omega^1_{<\Gamma(r)}(J_V) \times [\Delta, (1-\Delta)^2]^2$.

According to $(\ref{uniformU})$  and $(\ref{uniformV})$, $(U_n,V_n)$ converge uniformly in $r$. Therefore, the functions $f_n(\lambda,r) \equiv f(\lambda;(r,V_n))$ converge uniformly on $[\Delta,(1-\Delta)^2] \times ({r_-},r_+)$, and so do their unique zeros $\lambda_n(r)$. Arguing in a similar way, one can obtain that $\mu_n(r)$ converge uniformly on $({r_-},r_+)$.


To summarize, we have argued that the maps $r \mapsto (U_{n+1},V_{n+1},\lambda_{n},\mu_{n})$ are continuous, while the iterates $(U_{n+1},V_{n+1},\lambda_{n},\mu_{n})$ converge uniformly in $r$. This implies that the map
$$r \mapsto (U^*_r,V^*_r, \lambda(V^*_r,r), \mu(U^*_r,r))$$ 
is continuous  on $({r_-},r_+)$.  Since $\mu(U^*_{r_-},{r_-})>\lambda(V^*_{r_-},{r_-})$   and $\mu(U^*_{r_+},{r_+})<\lambda(V^*_{r_+},{r_+})$, the continuous functions $\lambda^*(V^*_r,r)$ and $\mu^*(U^*_r,r)$ must assume the same value at some point $r'$. We would like to emphasize, that we can not demonstrate that this point is unique.  We have for $r=r'$:
\begin{align*}
\lambda^\rho(V^*_{r'},r') \  U^*_{r'}&= U^*_{r'}     \circ \Psi_{V^*_{r'},\lambda(V^*_{r'},r'),r'},\\
\lambda^\rho(U^*_{r'},r') \ V^*_{r'}    &= V^*_{r'} \circ \Phi_{U^*_{r'},\lambda(U^*_{r'},r'),r'},
\end{align*}
on   $\field{C}_{J_U} \times \field{C}_{J_V}$. Recall, that $U_{r'}^*$ and  $V_{r'}^*$ are univalent on $\field{C}_{J_U}$ and $\field{C}_{J_V}$, respectively, if $U_0$ and $V_0$ are. This implies that the following Lorenz map is a renormalization fixed point of type  $(\{0,1\},\{1,0,0\})$:
$$(l \circ p_\rho, t \circ p_\rho),$$
where 
$$l(z)=r'-(U^*_{r'})^{-1}(z/a), \quad t(z)=(V^*_{r'})^{-1}(z/b) -1,$$
are analytic diffeomorphisms on 
$$a \circ U^*_{r'}(\fC_{J_{U}}), \quad a(z) \equiv a z,$$
and
$$b \circ V^*_{r'}(\fC_{J_{V}}), \quad b(z) \equiv b z,$$
respectively, and $a$ and $b$ are as in $(\ref{ab})$.
\end{proof}



\begin{thebibliography}{99}




\bibitem{CCT} P. Collet, P. Coullet and C. Tresser, {\it Scenarios under
  constraint}, J. Physique Lett. {\bf 46}(4) (1985), 143--147.
\bibitem{Eps1} H. Epstein, {\it New proofs of the existence of the Feigenbaum  functions},  Commun. Math. Phys. {\bf 106}(1986), 395--426.
\bibitem{Eps2} H. Epstein, {\it Fixed points of composition operators}, Nonlinear evolution and chaotic phenomena (Noto, 1987),  71--100, NATO Adv. Sci. Inst. Ser. B Phys., 176, {\it Plenum, New York}, 1988
\bibitem{Eps3} H. Epstein, {\it Fixed points of composition operators II},   Nonlinearity {\bf 2}(1989), 305--310.
\bibitem{GW} J. Guckenheimer and R. F. Williams, {\it Structural stability of the Lorenz attractors},  Publ. Math. IHES {\bf 50} (1979), 59--72.
\bibitem{HS} J. H. Hubbard, C. T. Sparrow, {\it The classification of topologically expansive Lorenz maps},  Comm. Pure Appl. Math.  {\bf 43}(4) (1990), 431--443.
\bibitem{MM} M. Martens and W. de Melo, {\it Universal models for Lorenz maps},  Ergod. Theor. and Dyn. Sys. {\bf 21}(3) (2001), 883--860.
\bibitem{Lor} E. N. Lorenz, {\it Deterministic non-periodic flow},  J. Atmos. Sci {\bf 20} (1963), 130--141.
\bibitem{Wil} R. F. Williams, {\it The structure of the Lorenz attractors},  Publ. Math. IHES {\bf 50} (1979), 73--79.
\bibitem{Win1} B. Winckler, {\it A renormalization fixed point for Lorenz map},  Nonlinearity {\bf 23}(6) (2010), 1291--1303.
\bibitem{Win2} B. Winckler, {Renormalization of Lorenz Maps},  Ph.D. Thesis, Institutionen f\"or Matematik, KTH, Stockholm, Sweden, (2011).



\end{thebibliography}
\end{document}